\documentclass{amsart}
\usepackage[utf8]{inputenc}

\title{Some combinatorial properties of splitting trees}
\author{Jonathan Schilhan}
\thanks{The author was partially supported through START-Project Y1012-N35 of the Austrian Science Fund, FWF, and by a UKRI Future Leaders Fellowship MR/T021705/1}
\usepackage{amsmath}
\usepackage{amssymb}
\usepackage{tikz-cd}
\usepackage{enumitem}
\usepackage{url}
\usepackage{hyperref}

    \DeclareMathOperator{\term}{term}
    
    \DeclareMathOperator{\wsacks}{wSacks}
    \DeclareMathOperator{\cont}{cont}
    \DeclareMathOperator{\baire}{Baire}
    \DeclareMathOperator{\borel}{Borel}

\theoremstyle{plain}
\newtheorem{thm}{Theorem}[section]
\newtheorem{lemma}[thm]{Lemma}

\theoremstyle{definition}
\newtheorem{definition}[thm]{Definition}

\newtheorem{remark}[thm]{Remark}

\begin{document}

\begin{abstract}
We show that splitting forcing does not have the weak Sacks property below any condition, answering a question of Laguzzi, Mildenberger and Stuber-Rousselle. We also show how some partition results for splitting trees hold or fail and we determine the value of cardinal invariants after an $\omega_2$-length countable support iteration of splitting forcing.  
\end{abstract}

\maketitle

\section{Introduction}

We will study some properties of \emph{splitting trees} and the associated \emph{splitting forcing} (see Definition~\ref{def:hittingtree}). This is a forcing notion that gives a natural way to add a \emph{splitting real} (see more below) generating a minimal extension of the ground model (see \cite[Corollary 4.21]{Schilhan2020} and also \cite{Spinas2007}). Splitting trees are part of a more general class of perfect trees that appeared in \cite[Definiton 2.11]{Shelah1992}. Recall that a set $T \subseteq \omega^{<\omega}$ is a \emph{tree} if it is closed under initial segments and that we denote with $[T]$ the set of all infinite branches through $T$. More precisely, $[T] = \{x \in \omega^\omega : \forall n \in \omega (x \restriction n \in T) \}$. The elements of $T$ are often called \emph{nodes}. A node $t \in T$ is called a \emph{terminal node} of $T$ if it has no proper extension in $T$, i.e. there is no $s \in T$ such that $t \subsetneq s$. The set of terminal nodes of $T$ is denoted $\term(T)$. We say that $t_0,t_1$ are incompatible, or write $t_0 \perp t_1$, to say that neither $t_0 \subseteq t_1$ nor $t_1 \subseteq t_0$. Finally, for any $t \in T$, we define the restriction of $T$ to $t$ as $T \restriction t = \{ s \in T : s \not\perp t \}$. 

\begin{definition}
Let $a \subseteq \omega$ and $X \subseteq 2^{\leq \omega}$. Then we say that $X$ \emph{covers} $a$ if for every $n \in a$ and $i \in 2$, there is $s \in X$ with $s(n) = i$. 
\end{definition}

\begin{definition}[Splitting tree]\label{def:hittingtree}
Let $T \subseteq 2^{<\omega}$ be a tree. Then $T$ is a \emph{splitting tree} if for every $t \in T$, $T \restriction t$ covers a cofinite subset of $\omega$. We write $\mathbb{SP}$ for the forcing notion consisting of splitting trees ordered by inclusion. 
\end{definition}

Recall that for $x,y \in [\omega]^\omega$, we say that $x$ \emph{splits} $y$ if $y \cap x$ as well as $y \setminus x$ are infinite. A real $x \in [\omega]^\omega$ is called \emph{splitting} over a model $V$ if for every $y \in V \cap [\omega]^\omega$, $x$ splits $y$. We will often identify sets of naturals with their characteristic functions in $2^\omega$. Thus we may sometimes say things such as $x \in 2^\omega$ splits $y \in [\omega]^\omega$, when we mean that $x^{-1}(1)$ splits $y$. Then we see that $\mathbb{SP}$ adds a generic splitting real over the ground model $V$. Namely, if $y \in [\omega]^\omega$, $n \in \omega$, $i \in 2$ and $T \in \mathbb{SP}$ are arbitrary, we find $s \in T$, such that $s(m)= i$ for some $m \in y \setminus n$. Thus, passing to the splitting tree $T \restriction s \leq T$, we force that the generic real intersects or avoids $y$ above $n$, depending on the choice of $i$.

Shelah showed in \cite{Shelah1992} that $\mathbb{SP}$ is proper and $\omega^\omega$-bounding. Moreover it has the continuous reading of names (see e.g. \cite{Spinas2007}). This is saying that for any $\mathbb{SP}$ name $\dot y$ for an element of $\omega^\omega$ and a condition $T \in \mathbb{SP}$, there is $S \leq T$ and a continuous function $f \colon [S] \to \omega^\omega$ such that $S \Vdash \dot y = f(\dot x_{\operatorname{gen}})$, where $\dot x_{\operatorname{gen}}$ is a name for the generic real added by $\mathbb{SP}$. Together with the minimality of the forcing extension this makes splitting forcing very similar to Sacks forcing. In \cite{LMSR}, the authors ask the natural question of whether splitting forcing also has the \emph{Sacks property} (see Definition~\ref{def:sacks} below). This alone can in fact already be answered using an earlier result of Zapletal (see Remark~\ref{rem:zapsacks}). But we will show in Section 2 that $\mathbb{SP}$ does not even have the \emph{weak Sacks property} (Definition~\ref{def:weaksacks}) below any of its conditions. 

The term ``splitting tree" probably first appeared in \cite{Spinas2004}, where they were introduced to give a topological characterization of analytic \emph{$\omega$-splitting families}, similar to the existing ones, e.g. for unbounded families (see \cite{Kechris1977}) or dominating families (see \cite{Spinas1994, Brendle1995}). Recall that a \emph{splitting family} is a set $\mathcal{S} \subseteq [\omega]^\omega$ so that for every $y \in [\omega]^\omega$ there is $x \in \mathcal{S}$ splitting $y$. Moreover $\mathcal{S} \subseteq [\omega]^\omega$ is called \emph{$\omega$-splitting} if for every countable $H \subseteq [\omega]^\omega$, there is $x \in \mathcal{S}$ simultaneously splitting every member of $H$. Again, identifying subsets of $\omega$ with characteristic functions, we may talk about subsets of $2^\omega$ being splitting or $\omega$-splitting. The main result of \cite{Spinas2004} is the following. 

\begin{thm}[{\cite[Theorem 1.2]{Spinas2004}}]\label{thm:analyticsplitting}\label{thm:spinas}
An analytic set $A \subseteq 2^\omega$ is $\omega$-splitting if and only if it contains the branches of a splitting tree. 
\end{thm}

This is a fundamental property of splitting trees. Essentially, it is showing that splitting forcing is an \emph{idealized forcing} in the sense of \cite{Zapletal2008} in which closed sets are dense\footnote{That $\mathbb{SP}$ is an idealized forcing in fact already follows from \cite[Prop. 2.1.6]{Zapletal2008}.}. Namely, let $\mathcal{I} \subseteq \mathcal{P}(2^\omega)$ be the $\sigma$-ideal consisting of sets that are not $\omega$-splitting. Then Theorem~\ref{thm:analyticsplitting} shows that $\mathbb{SP}$ is equivalent in the sense of forcing to the partial orders of $\mathcal{I}$-positive analytic, Borel or closed subsets of $2^\omega$. 

Later, in \cite{Spinas2007}, Spinas studied properties of splitting trees related to Borel subset-colorings. The main result\footnote{Spinas' result also follows immediately from Lemma~\ref{lem:mcg} below.} of \cite{Spinas2007}, is that $$ 2^{<\omega} \not\rightarrow_{\borel} [\mathbb{SP}]^2_{2^{\aleph_0}}.$$ Here, we adopt the well-known arrow notation from partition calculus. The above is then saying that for any Borel map $c \colon [2^\omega]^2 \to 2^\omega$, there is some splitting tree $T$ such that $c''[[T]]^2 \subsetneq 2^\omega$, i.e. we can avoid at least one color on the branches of a splitting tree\footnote{The outer pair of brackets in $[[T]]^2$ corresponds to the usual notation for the collection of two-sized subsets $[X]^2$ of a set $X$.}. Whenever $T \subseteq 2^{<\omega}$ is a tree, $\Gamma$ is some class of pair-colorings on $[T]$ (such as the continuous, Borel or Baire-measurable colorings) and $\mathbb{P}$ is a set of trees, we write $$T \rightarrow_\Gamma (\mathbb{P})^2_j$$ to say that for every coloring $c \colon [[T]]^2 \to j$ that is in $\Gamma$, there is $S \subseteq T$, $S \in \mathbb{P}$ such that $c$ is constant on $[[S]]^2$. For example, when $\mathbb{S}$ is \emph{Sacks forcing}, i.e. the collection of all perfect subtrees of $2^{<\omega}$, then $$T \rightarrow_{\baire} (\mathbb{S})^2_{j},$$ for $T$ an arbitrary perfect tree and $j \in \omega$, is known as Galvin's Theorem (see \cite[Theorem 19.6]{Kechris1995}). This is generally wrong for splitting trees as we shall see in Section~\ref{sec:Partition}. In fact there is a dense set of splitting trees $T \in \mathbb{SP}$ so that $$ T \not\rightarrow_{\cont} (\mathbb{SP})^2_2.$$ Combined with Theorem~\ref{thm:analyticsplitting}, this is showing that there are continuous pair-colorings on the branches of a splitting tree without a homogeneous $\omega$-splitting subfamily.

On the other hand $$2^{<\omega} \rightarrow_{\baire} (\mathbb{SP})^2_j $$ holds for every $j \in \omega$. 

In Section 4, we will study the model obtained by iterating $\mathbb{SP}$ in a countable support iteration of length $\omega_2$. We will show how to decide the value of the classical cardinal invariants in Cichoń's diagram (see \cite{BartoszynskiJudah1995} for a reference) and the combinatorial ones appearing in \cite{VanDouwen1984} or \cite{Blass2010}. 

\section{The Sacks property}

Let us recall the Sacks property and the weak Sacks property. 

\begin{definition}\label{def:sacks}
A forcing $\mathbb{P}$ has the \emph{Sacks property} if for every $g \colon \omega \to \omega \setminus 1$ such that $g(n) \longrightarrow \infty$, every condition $p \in \mathbb{P}$ and every name $\dot x$ for an element of $\omega^\omega$, there is $q \leq p$ and $H \in \prod_{n \in \omega} [\omega]^{g(n)}$, such that $$ q \Vdash \forall n \in \omega (\dot x(n) \in H(n)).$$
\end{definition}

\begin{definition}\label{def:weaksacks}
A forcing $\mathbb{P}$ has the \emph{weak Sacks property} if for every $g \colon \omega \to \omega$ such that $g(n) \longrightarrow \infty$, every condition $p \in \mathbb{P}$ and every name $\dot x$ for an element of $\omega^\omega$, there is $q \leq p$, $a \in [\omega]^\omega$ and $H \in \prod_{n \in a} [\omega]^{g(n)}$, such that $$ q \Vdash \forall n \in a (\dot x(n) \in H(n)).$$
\end{definition}

\begin{definition}
Let $T$ be a splitting tree. Then we say that $T$ is \emph{finitely covering} if there is a finite set $X \subseteq [T]$ that covers a cofinite subset of $\omega$. 
\end{definition}

\begin{lemma}\label{lem:main}
Let $T$ be not finitely covering and $l_0, m \in \omega$. Then there is $l_1 > l_0$ so that no $X \subseteq T \cap 2^{l_1}$ of size $m$ can cover $[l_0, l_1)$. 
\end{lemma}

\begin{proof}
Suppose to the contrary that for every $l > l_0$ there is $(y_j^l)_{j<m} \in [T]^m$ so that $\{y_j^l : j < m\}$ covers $[l_0, l)$. By compactness, there is $a \in [\omega]^\omega$ so that $\langle (y_j^l)_{j<m} : l \in a \rangle$ converges to $(y_j)_{j < m} \in [T]^m$. Since $T$ is not finitely covering there is $l > l_0$ and $i \in 2$ so that for every $j < m$, $y_j(l)= i$. On the other hand, there is $l' \in a \setminus (l+1)$ so that $y_j^{l'} \restriction (l+1) \subseteq y_j$ for every $j < m$. Thus $y_j^{l'}(l) = i$ for every $j < m$ and $\{y_j^{l'} : j < m\}$ does not cover $[l_0, l')$.
\end{proof}

\begin{lemma}
\label{lem:nonfincov}
Let $T$ be not finitely covering and $g \colon \omega \to \omega$ such that $g(n) \longrightarrow \infty$. Then there is a continuous function $f \colon [T] \to \omega^\omega$ so that for any subtree $S \subseteq T$, $a \in [\omega]^\omega$ and $H \in \prod_{n \in a} [\omega]^{\leq g(n)} \times \omega^{\omega \setminus a}$, if $$f''[S] \subseteq \prod_{n \in \omega} H(n),$$ then $S$ is not a splitting tree. 
\end{lemma}

\begin{proof}
Using Lemma~\ref{lem:main}, let $\langle l_n : n \in \omega \rangle$ be a sequence so that for every $n \in \omega$, no set $X \subseteq T \cap 2^{l_{n+1}}$ of size $g(n)$ can cover $[l_{n}, l_{n+1})$. Now simply let $f(x)(n) = m$ iff $x \restriction l_{n+1}$ is the $m$'th element of $T \cap 2^{l_{n+1}}$ in lexicographical order. Then, whenever $S \leq T$ is such that at most $g(n)$ values can be attained as $f(x)(n)$ for $x \in [S]$, $ \vert S \cap 2^{l_{n+1}} \vert \leq g(n)$ and $S$ does not cover $[l_n,l_{n+1})$. If infinitely many intervals $[i_n, i_{n+1})$ are not covered then $S$ is not a splitting tree. 
\end{proof}

In order to get a failure of the weak Sacks property it is thus sufficient to find a splitting tree that is not finitely covering. Such a tree is not hard to construct directly, but the next lemma will produce such a tree below any given condition. 

\begin{lemma}\label{lem:mcg}
Let $T$ be a splitting tree, and $T \in M$ for $M$ a countable transitive model of a large enough fragment of ZFC. Then there is $S \leq T$ a splitting tree so that for any pairwise distinct $x_0, \dots, x_{n-1} \in [S]$, $$(x_0, \dots, x_{n-1}) \text{ is } T^n\text{-generic over } M.$$
\end{lemma}

Here, we view $T^n$ as the forcing ordered by coordinate-wise extension, which is equivalent to Cohen forcing. 

\begin{proof}
This is Proposition~4.16 in \cite{Schilhan2020} when applied to $k=1$, in combination with Lemma~4.10 (also see Definition~4.6 and 4.9).
\end{proof}

We do not want to reprove Lemma~\ref{lem:mcg} here since it would result in essentially copying the exact argument from \cite{Schilhan2020}, which can be read independently from the rest of the paper. The following is a simple genericity argument.

\begin{lemma}\label{lem:mcgcover}
Let $S$ be as in Lemma~\ref{lem:mcg} for arbitrary $T$ and $M$. Then $S$ is not finitely covering. 
\end{lemma}

Finally we get: 

\begin{thm}
Splitting forcing does not have the weak Sacks property below any condition. 
\end{thm}

At first sight, Lemma~\ref{lem:mcg} seems to be overkill to find a subtree that is not finitely covering. But we are not aware of a simpler proof that is not essentially the same combinatorial argument used in the proof of Lemma~\ref{lem:mcg}, which revolves around having to deal with arbitrary tuples of nodes, rather than just extending single nodes along the construction of a splitting subtree. 

In the next theorem we characterize splitting trees on which a construction such as in Lemma~\ref{lem:nonfincov} is possible and we shall see that finite-covering is an essential idea. Moreover, we will see that it does not matter whether we consider Baire-measurable functions or just continuous ones. For a splitting tree $T$ and a class $\Gamma$ of functions $f \colon [T] \to \omega^\omega$ we will write $\wsacks_\Gamma(T)$ to say that the definition of the weak Sacks property holds applied to the condition $p = T$ and a name $\dot x$ for $f(\dot x_{\operatorname{gen}})$ where $f \in \Gamma$ and $\dot x_{\operatorname{gen}}$ is a name for the generic real. 

\begin{thm}
\label{thm:wsacks}
Let $T$ be a splitting tree. Then the following are equivalent. 
\begin{enumerate}
    \item $\wsacks_{\baire}(T)$
    \item $\wsacks_{\cont}(T)$
    \item $\exists S \leq T \forall s \in S(S\restriction s \text{ is finitely covering})$
\end{enumerate}
\end{thm}

\begin{proof}

$(3)\rightarrow(1)$: Let $g \in \omega^\omega$ be so that $g(n)$ diverges to $\infty$ and let $f \colon [T] \to \omega^\omega$ be Baire-measurable. Then there are open dense sets $O_n \subseteq [T]$, for $n\in \omega$, so that $f$ is continuous on $X := \bigcap_{n \in \omega} O_n$. By $(3)$ we can assume wlog that $T$ is such that for every $s \in T$, $T\restriction s$ is finitely covering. For every $s \in T$, fix $x^s_0, \dots, x^s_{k-1}$ in $[T \restriction s]$ and $z^s \in [T \restriction s] \cap X$, all pairwise distinct, $k(s) = k \in \omega$ and $m(s)$ so that $\{x^s_0, \dots, x^s_{k-1}\}$ covers $[m(s), \omega)$. We are going to construct a sequence $\langle T_n : n \in \omega \rangle$ of finite subtrees of $T$ that is an strictly increasing by end-extension, such that $S := \bigcup_{n \in \omega} T_n$ is a splitting tree witnessing the weak Sacks property for $f$ and $g$. Along the construction, there will be a map $\sigma$ that maps every $s \in \term(T_n)$ to a dedicated ``working node" $\sigma(s) \in \term(T_{n+1})$ and there will be a strictly increasing sequence $\langle i_n : n \in \omega \rangle$ of naturals. The following properties will be satisfied for every $n \in \omega$: 

\begin{enumerate}[label=(\alph*)]
    \item For every $s \in \term(T_{n+1})$, $[T \restriction s] \subseteq O_n$,
    \item $s$ determines $f(x)(i_n)$ for every $x \in [T\restriction s] \cap X$ and $\vert \term(T_{n+1}) \vert < g(i_n)$. 
    \item For every $s \in \term(T_n)$, $T_{n+1} \restriction s$ covers $[m(s), m(\sigma(s))$.  
\end{enumerate}

(a) implies that $[S] \subseteq X$ and (b) then yields that $S$ witnesses $\wsacks(T)$ for $f$ and $g$. (c) on the other hand implies that for every $s \in \term(T_n)$, $S \restriction n$ covers $[m(s), \omega)$ and thus that $S$ is a splitting tree. 

The construction is as follows. We start with $T_0 = \{ \emptyset \}$ and $i_0 = 0$. Next, given $T_n$, we let $K = \sum_{s \in \term(T_n)} (1+ k(s))$. Since the values of $g$ diverge to infinity, there is $i_n > i_{n-1}$ so that $g(i_n) > K$. For each $s \in \term(T_n)$, let $\sigma(s) := z^s \restriction l$ for $l$ large enough such that $[T \restriction (z^s \restriction l) ] \subseteq O_n$, $z^s \restriction l$ determines the value of $f(x)(i_n)$ on $X$ and $z^s\restriction l$ is not an initial segment of any of the $x^s_i$, for $i<k(s)$. Let $m = \max \{ m(\sigma(s)) : s \in \term(T_n)\}$ and for every $s \in \term(T_n)$ and $i < k(s)$, find $t^s_i$ extending $x^s_i \restriction m$ such that $[T \restriction t^s_i] \subseteq O_n$ and $t^s_i$ decides the value $f(x)(i_n)$ on $X$. Finally let $T_{n+1}$ be the downwards closure of $\bigcup_{s \in \term(T_n)} \{t^s_0, \dots, t^s_{k(s)-1} \} \cup  \{ \sigma(s) : s \in \term(T_n) \}$.

$(2) \rightarrow (3)$: Assume that $(3)$ fails. For a splitting tree $S$ consider a pruning derivative $S' := \{s \in S : S \restriction s \text{ is finitely covering} \}$. Now define $T_0 = T$, $T_{\alpha +1} = T_\alpha'$ and $T_\gamma = \bigcap_{\beta < \gamma} T_\beta $ for every $\alpha < \omega_1$ and $\gamma < \omega_1$ a limit ordinal. Since we assume that $(3)$ fails, we have that for every $\alpha < \omega_1$, if $T_{\alpha} \neq \emptyset$, then $T_{\alpha} \setminus T_{\alpha +1} \neq \emptyset$. In particular there is some $\alpha < \omega_1$ so that $T_\alpha = \emptyset$.

\begin{lemma}
Let $\langle X_n : n \in \omega \rangle$ be a partition of $2^\omega$ into closed sets and $f_n \colon X_n \to \omega^\omega$ be continuous, for every $n \in \omega$. Then there is a continuous function $f \colon 2^\omega \to \omega^\omega$ so that for every $n \in \omega$, $f(x)(m) = f_n(x)(m)$ for all $m \geq n$ and $x \in X_n$. 
\end{lemma}

\begin{proof}
Let $X_n = [R_n]$ for a perfect tree $R_n$, for every $n \in \omega$. We define a sequence $\langle \varphi_n : n \in \omega \rangle$, where $\varphi_n \colon 2^{\leq i_n} \to \omega^{\leq n}$ is order preserving for every $n$ and $\langle i_n : n \in \omega \rangle$ is increasing, inductively as follows. Start with $\varphi_0 = \{ (\emptyset, \emptyset) \}$ and $i_0 = 0$. Given $\varphi_n$ and $i_n$, let $i_{n+1} > i_n$ be large enough so that $R_{i} \cap R_{j} \cap 2^{i_{n+1}} = \emptyset$  and for every $x \in X_i$, $f_i(x)(n)$ is decided by $x \restriction i_{n+1}$, for every $i,j < n$. Let $\varphi_{n+1} \supseteq \varphi_n$, $\varphi_{n+1} \colon 2^{\leq i_{n+1}} \to \omega^{\leq n}$ be arbitrary so that $\varphi_{n+1}(s) \subseteq f_i(x)$ for $s \in R_i \cap 2^{i_{n+1}}$, $x \in X_i$ and $i < n$. Finally, let $\varphi = \bigcup_{n \in \omega} \varphi_n$ and $$f(x) := \bigcup_{n \in \omega} \varphi(x \restriction n).$$ We see that $f$ is as required. 
\end{proof}

For every $\beta < \alpha$, let $A_\beta$ be the set of minimal $s \in T_{\beta} \setminus T_{\beta+1}$. For any $s \in A_\beta$, applying Lemma~\ref{lem:nonfincov} we find a continuous function $f_{\beta,s} \colon [T_\beta \restriction s] \to \omega^\omega$ so that for any splitting tree $S \leq T_\beta \restriction s$, for almost all $n \in \omega$, more than $2^n$ values can be obtained as $f_{\beta,s}(x)(n)$ for $x \in [S]$. Apply the previous Lemma to find a continuous function $f \colon [T] \to \omega^\omega$ so that $f(x) =^* f_{\beta,s}(x)$ uniformly for all $x \in [T_\beta \restriction s]$. We  claim that $f$ is as required. Namely, suppose that there is $S \leq T$ so that for infinitely many $n \in \omega$, $\vert \{ f(x)(n) : x \in [S] \} \vert < 2^n$. Then there is $\beta < \alpha$ and $s \in \omega$ so that $[S] \cap [T_{\beta} \restriction s]$ is $\omega$-splitting and in particular contains the branches of a splitting tree. This is a contradiction. 

$(1) \rightarrow (2)$ is clear. 
\end{proof}

It would be interesting to find a similar characterization for the Sacks property. 

\section{Partition results}
\label{sec:Partition}

\begin{thm}
\label{thm:wsacksandcolor}
Let $T$ be a splitting tree and assume that $T \rightarrow_{\cont} (\mathbb{SP})^2_{2}$. Then $\wsacks_{\cont}(T)$ is satisfied. 
\end{thm}

\begin{proof}
Let $f \colon [T] \to \omega^\omega$ be continuous and $g \colon \omega \to \omega$ diverge to $\infty$. Define an increasing sequence $\langle i_n : n \in \omega \rangle$ recursively such that for every $n \in \omega$, there is $m > i_n$ with $g(m) > 2^{i_n}$ and $f(x)(m)$ is decided by $x \restriction i_{n+1}$ for every $x \in [T]$. Define a coloring $c \colon [T]^2 \to 2$ such that $$c(x,y) = \begin{cases} 0 &\text{ if } \exists k \in \omega (\Delta(x,y) \in [i_{2k}, i_{2k+1})) \\
1 &\text{ if } \exists k \in \omega (\Delta(x,y) \in [i_{2k+1}, i_{2k+2}))
\end{cases},$$ where $\Delta(x,y) = \min \{l \in \omega : x(l)\neq y(l) \}$. Now let $S \leq T$ be such that $[S]$ is $c$-homogeneous, say with constant color $0$. Then we see that for every $k \in \omega$, $\vert S \cap 2^{i_{2k+2}} \vert \leq 2^{i_{2k+1}}$. Thus there is $m > i_{2k+1}$ such that at most $2^{i_{2k+1}} < g(m)$ many values can be attained as $f(x)(m)$ for $x \in [S]$. Similarly for color $1$. \end{proof}

From Theorem~\ref{thm:wsacksandcolor} and the results of the last section we immediately find that: 

\begin{thm}
\label{thm:nohomo}
Let $T$ be a splitting tree. Then there is a splitting tree $S \leq T$ so that $$S\not\rightarrow_{\cont} (\mathbb{SP})^2_{2}. $$
\end{thm}

Our motivation behind proving this theorem were the results of \cite{Schrittesser2016} which use Galvin's Theorem for Sacks forcing and an analogue for its countable support iterations (see \cite[Theorem 3.17]{Schrittesser2016}) in a crucial way to get certain canonization results for binary relations. These results were vastly generalized in \cite{Schilhan2020} using a technique based on Lemma~\ref{lem:mcg}, to be applied, among other things, to splitting forcing. The above theorem shows that the method of \cite{Schrittesser2016} does not apply in a similar way to splitting forcing. In its simplest form, the canonization result says that for any anayltic equivalence relation $E$ on $2^\omega$ and any perfect tree $T$, there is a perfect subtree $S \subseteq T$ so that $E$ is canonical on $[S]$, i.e. either $[S]$ consists of pairwise $E$-inequivalent reals or $[S]$ is contained in an $E$-equivalence-class. This is imediate from Galvin's Theorem. For splitting trees on the other hand, we apply Lemma~\ref{lem:mcg} to $T$ and a model $M$ containing a code of $E$ to get $S \leq T$ consisting of mutual $T$-generics over $M$. If $x,y \in [S]$ are $E$-equivalent, then there is $s \subseteq x$ forcing over $M[y]$ that $x E y$. In particular, for every $x \in [S \restriction s]$, $x E y$ and $S \restriction s$ is contained in an $E$-class.

\begin{thm}
\label{thm:trivialcolor}
Let $c \colon [2^\omega]^2 \to j$ be Baire-measurable, $j \in \omega$. Then there is a splitting tree $T$, so that $[T]$ is $c$-homogeneous. 
\end{thm}

\begin{proof}
In the following, we will identify $c$ with the corresponding symmetric function on $(2^\omega)^2$. Since $c$ is Baire-measurable, there is a decreasing sequence $(O_n)_{n \in \omega}$ of dense open subsets of $(2^\omega)^2$ such that $c$ is continuous on $X = \bigcap_{n\in \omega} O_n$. Fix $i \in j$ and $s \in 2^{<\omega}$ so that for any $t \in 2^{<\omega}$, $s \subseteq t$, there are incompatible $t_0,t_1$ extending $t$ such that $c$ is constant with color $i$ on $([t_0] \times [t_1]) \cap X$. Without loss of generality we may assume that $s = \emptyset$. We will construct an increasing sequence $(T_n)_{n \in \omega}$ of finite subtrees of $2^{<\omega}$ and an increasing sequence $(m_n)_{n \in \omega}$ of natural numbers with the following properties for every $n \in \omega$.

\begin{enumerate}
    \item All terminal nodes of $T_n$ have length $m_n$,
    \item for any terminal nodes $t_0 \neq t_1$ in $T_n$, $c$ has constant value $i$ on $([t_0] \times [t_1]) \cap X$ and $[t_0] \times [t_1] \subseteq O_n$, 
    \item every $t \in T_n$ has at least $4$ pairwise incompatible extensions in $T_{n+1}$,
    \item for every $t \in T_n$, $\{ t' \in T_{n+2} : t \subseteq t' \}$ covers $[m_{n+1}, m_{n+2})$. 
\end{enumerate}

Provided we have constructed such a sequence, we shall check that $T = \bigcup_{n \in \omega} T_n$ is as required. It is clear from (2) that $[T]$ is $c$-homogeneous with color $i$. To show that $T$ is a splitting tree, let $t \in T$ be arbitrary, say without loss of generality that $t$ is a terminal node of $T_n$ for some $n \in \omega$. Moreover let $l \in [m_{k},m_{k+1})$ for some $k \geq n+1$. Then we may extend $t$ to $t^+$ in $T_{k-1}$ and by (4) there are $t_0, t_1$ extending $t^+$ in $T$ so that $t_0(l) = 0$ and $t_1(l) = 1$. Thus we have shown that $T_t$ covers $[m_{n+1}, \omega)$.  

Now let us construct the sequence recursively. We start with $T_0$ and $m_0 \in \omega$ so that (1) and (2) are satisfied. To find such $T_0$ and $m_0$, let $t_0 \perp t_1$ be arbitrary such that $c$ has constant color $i$ on $([t_0]\times [t_1]) \cap X$ . Next extend $t_0$ to $t_0^0$ and $t_1^0$, such that $t_0^0 \perp t_1^0$ and $c$ has constant color $i$ on $([t^0_0]\times [t^0_1]) \cap X$ and similarly find $t_0^1, t_1^1$. Finally, extend $t^0_0,t^0_1,t^1_0$ and $t^1_1$ to nodes of the same length $m_0$ such that the open subsets of $(2^\omega)^2$ determined by any pair of these nodes are contained in $O_0$. Finally let $T_0$ be the tree generated by these nodes. Given $T_n$ and $m_n$, we proceed as follows. We construct an increasing sequence of finite trees $(R_k)_{k < K}$ of some unspecified finite length $K$, starting with $R_0 = T_n$. For every $k < K$, all terminal nodes of $R_k$ will have the same length. In each step $k < K$, there is a working node $t \in \term(R_k)$, that has exactly two incompatible extensions $t_0, t_1$ which are terminal nodes in $R_{k+1}$ such that $c$ has color $i$ on $([t_0]\times [t_1]) \cap X$. All other terminal nodes $t'$ of $R_k$ are extended uniquely to terminal nodes of the form $t'^\frown 0^\frown \dots^\frown 0$ or $t'^\frown 1^\frown \dots^\frown 1$ in $R_{k+1}$. This is done in a way that for every $r \in T_{n-1}$, there is one such $t'$ extending $r$ which is extended by $0$'s and another one that is extended by $1$'s. This is possible since every $r \in T_{n-1}$ is extended by at least $3$ (in fact $4$) pairwise incompatible nodes in $T_n$, so at most $1$ of them is the working node $t$. In case $n=0$, this step is irrelevant. By varying the working nodes we can easily ensure in finitely many steps $k < K \in \omega$ that all terminal nodes of $T_n$ are extended by at least $4$ pairwise incompatible nodes in $R_{K-1}$. Next we construct an increasing sequence $(S_l)_{l<L}$ of unspecified finite length $L$, starting with $S_0 = R_{K-1}$. Again all terminal nodes of $S_l$ have the same length for each $l<S$. This time, in every step $l$ there is a working pair $(u,v) \in \term(S_l)^2$. In $\term(S_{l+1})$, $u$ and $v$ are extended uniquely to $u'$ and $v'$ respectively such that $[u'] \times [v'] \subseteq O_{n+1}$. All other terminal nodes of $S_l$ are extended by $0$'s only or by $1$'s only in the same way as before, now using that at least $4$ pairwise incompatible nodes extend each $r \in T_{n-1}$. After finitely many steps, we can ensure that every $(u,v) \in \term(S_0)$ is extended by $u'$ and $v'$ respectively such that $[u'] \times [v'] \subseteq O_{n+1}$. Finally, we let $T_{n+1} = S_{L-1}$. 
\end{proof}

\section{The splitting model}

\begin{thm}
Let $\mathbb{P}$ be the $\omega_2$-length countable support iteration of $\mathbb{SP}$. Then, in $V^\mathbb{P}$, $\aleph_1 = \operatorname{cof}(\mathcal{M}) = \mathfrak{a} < \mathfrak{r} = \operatorname{non}(\mathcal{N}) = \mathfrak{c} = \aleph_2$.
\end{thm}

\begin{lemma}
Let $\mathbb{P}$ be a countable support iteration of $\mathbb{SP}$, $p \in \mathbb{P}$ and $\dot y$ a $\mathbb{P}$-name for a real. Moreover let $A \in V$ be arbitrary. Then there is a countable elementary submodel $M \preccurlyeq H(\theta)$, for large enough $\theta$, with $A \in M$, $q \leq p$ and a name $\dot c$, so that $$q \Vdash`` \dot c \text{ is a Cohen real over } M \text{ and } \dot y \in M[\dot c]".$$
\end{lemma}

\begin{proof}
This follows from a much stronger result proved in \cite[Lemma 4.22]{Schilhan2020}. See also the start of Section 4.3 and Lemma 2.2 in \cite{Schilhan2020}.
\end{proof}

\begin{proof}[Proof of Theorem~4.1]
Note first, that we can assume wlog that $V \models \operatorname{CH}$ since this is forced in the first $\omega_1$-many steps of any csi of nontrivial forcings. Then it is clear that $V^\mathbb{P} \models \mathfrak{d} = \aleph_1$ as $\mathbb{SP}$ is $\omega^\omega$-bounding. Next, we show that $V \cap 2^\omega$ is non-meager in $V^\mathbb{P}$. Suppose to the contrary that $V \cap 2^\omega$ is forced to be contained in a meager $F_\sigma$ set $\dot X$ by a condition $p \in \mathbb{P}$. $\dot X$ is going to be coded by a real $\dot y$ and by the previous lemma we find $q \leq p$, $M \preccurlyeq H(\theta)$ and $\dot c$ so that $$q \Vdash`` \dot c \text{ is a Cohen real over } M \text{ and } \dot y \in M[\dot c]".$$ But Cohen forcing preserves the set of ground model reals to be non-meager. In particular, letting $G$ be $\mathbb{P}$-generic over $V$ with $q \in G$, we have that $$M[\dot c[G]] \models `` \text{The } F_\sigma \text{ set coded by } \dot y \text{ does not cover the reals of } M". $$ 
By an absoluteness argument between $M[\dot c[G]]$ and $V[G]$, we find that there is a real $x \in M \subseteq V$ so that $x \notin \dot X[G]$, yielding a contradiction. 
Now that we have that $\mathfrak{d} = \operatorname{non}(\mathcal{M}) = \aleph_1$ in $V^\mathbb{P}$, we may also follow that $\operatorname{cof}(\mathcal{M}) = \aleph_1$ since $\operatorname{cof}(\mathcal{M}) = \max (\mathfrak{d},\operatorname{non}(\mathcal{M}))$ (see e.g. \cite[Theorem~2.2.11]{BartoszynskiJudah1995}).

The argument for showing that $\mathfrak{a} = \aleph_1$ is similar to the one we just made for $\operatorname{non}(\mathcal{M})$. Namely, we show that any Cohen-indestructible mad family is also $\mathbb{P}$-indestructible. To this end, let $\mathcal{A}$ be a mad family so that $\Vdash_{\mathbb{C}} ``\mathcal{A} \text{ is mad}"$. Towards a contradiction, suppose that $\dot y$ is a $\mathbb{P}$ name for an infinite subset of $\omega$ and that $p \Vdash \forall x \in \mathcal{A} (\vert x \cap y \vert < \omega)$. Then, once again, there is an elementary enough $M$, this time with $\mathcal{A} \in M$, a name $\dot c$ and $q \leq p$ so that $q$ forces that $\dot c$ is Cohen generic over $M$ and that $M[\dot c]$ contains $\dot y$. But then there must be $x \in \mathcal{A} \cap M$ so that $\vert y \cap x \vert = \omega$. Cohen-indestructible mad families exist under CH (see \cite[Exercise IV.4.12]{Kunen2011}). 

The reaping number $\mathfrak{r}$ is large in $V^\mathbb{P}$ as $\mathbb{SP}$ adds splitting reals. Finally, it was shown in \cite[Proposition 4.1.29]{Zapletal2008} that there is a condition $p \in \mathbb{SP}$ forcing that the set of ground model reals is null. This easily implies that in $V^{\mathbb{P}}$, $\operatorname{non}(\mathcal{N}) = \aleph_2$. 
\end{proof}

\begin{remark}
The arguments for $\operatorname{non}(\mathcal{M})$ and $\mathfrak{a}$ above are an instance of a more general result that shows that \emph{very tame invariants} (see \cite[Definition 6.1.9]{Zapletal2008}) that Cohen forcing keeps small, stay small after forcing with $\mathbb{SP}$. More generally, it is not hard to apply the theory developed in \cite[6.1]{Zapletal2008} to the $\sigma$-ideal of non $\omega$-splitting sets of reals to get similar results for other forcing notions adding $\omega$-splitting reals.
\end{remark}

\begin{remark}
\label{rem:zapsacks}
The fact that some condition in $\mathbb{SP}$ forces the ground model reals to be null already shows that below this condition $\mathbb{SP}$ cannot have the Sacks property. This is because the Sacks property keeps the ground model reals non-null (in fact it keeps $\operatorname{cof}(\mathcal{N})$ small, see \cite[2.3]{BartoszynskiJudah1995}).
\end{remark}

Other than the results above, we can also show that P-points exist in $V^\mathbb{P}$ (\cite[Section 4.5.2]{Schilhan2020a}). This uses the technique developed in \cite{Schilhan2020} in an essential way. The countable support iteration of Sacks forcing, for instance, outright preserves ground model P-points, but this is an impossible task for a forcing adding splitting reals. Rather, $\mathbb{P}$ can preserve the union of $\aleph_1$ many ground model Borel sets to be reinterpreted as a P-point in the extension. To our best knowledge, this is the only existence result for P-points of this kind. 

\bibliographystyle{plain}

\begin{thebibliography}{10}

\bibitem{BartoszynskiJudah1995}
Tomek Bartoszy\'nski and Haim Judah.
\newblock {\em Set Theory: On the Structure of the Real Line}.
\newblock Ak Peters Series. Taylor \& Francis, 1995.

\bibitem{Blass2010}
Andreas Blass.
\newblock {\em Combinatorial Cardinal Characteristics of the Continuum}, pages
  395--489.
\newblock Springer Netherlands, Dordrecht, 2010.

\bibitem{Brendle1995}
Jörg Brendle, Greg Hjorth, and Otmar Spinas.
\newblock Regularity properties for dominating projective sets.
\newblock {\em Annals of Pure and Applied Logic}, 72(3):291--307, 1995.

\bibitem{Kechris1977}
Alexander~S. Kechris.
\newblock On a notion of smallness for subsets of the baire space.
\newblock {\em Transactions of the American Mathematical Society},
  229:191--207, 1977.

\bibitem{Kechris1995}
Alexander~S. Kechris.
\newblock {\em Classical descriptive set theory}, volume 156 of {\em Graduate
  Texts in Mathematics}.
\newblock Springer-Verlag, New York, 1995.

\bibitem{Kunen2011}
Kenneth Kunen.
\newblock {\em Set Theory (Studies in Logic: Mathematical Logic and
  Foundations)}.
\newblock College Publications, 2011.

\bibitem{LMSR}
Giorgio Laguzzi, Heike Mildenberger, and Brendan Stuber-Rousselle.
\newblock On splitting trees.
\newblock \url{https://arxiv.org/abs/2004.10840}, 2020.

\bibitem{Schilhan2020a}
Jonathan Schilhan.
\newblock {\em Combinatorics and Definability on the Real Line and the Higher
  Continuum}.
\newblock PhD thesis, 2020.

\bibitem{Schilhan2020}
Jonathan Schilhan.
\newblock Tree forcing and definable maximal independent sets in hypergraphs.
\newblock \url{https://arxiv.org/abs/2009.06445}, 2020.

\bibitem{Schrittesser2016}
David Schrittesser.
\newblock Definable discrete sets with large continuum.
\newblock \url{https://arxiv.org/abs/1610.03331}, 2016.

\bibitem{Shelah1992}
Saharon Shelah.
\newblock {Vive la Diff{\'e}rence I: Nonisomorphism of Ultrapowers of Countable
  Models}.
\newblock In Haim Judah, Winfried Just, and Hugh Woodin, editors, {\em Set
  Theory of the Continuum}, pages 357--405, New York, NY, 1992. Springer US.

\bibitem{Spinas1994}
Otmar Spinas.
\newblock Dominating projective sets in the baire space.
\newblock {\em Annals of Pure and Applied Logic}, 68(3):327--342, 1994.

\bibitem{Spinas2004}
Otmar Spinas.
\newblock {Analytic countably splitting families}.
\newblock {\em Journal of Symbolic Logic}, 69(1):101 -- 117, 2004.

\bibitem{Spinas2007}
Otmar Spinas.
\newblock Splitting squares.
\newblock {\em Israel Journal of Mathematics}, 162(1):57--73, December 2007.

\bibitem{VanDouwen1984}
Eric~K. {Van Douwen}.
\newblock The integers and topology.
\newblock In Kenneth Kunen and Jerry~E. Vaughan, editors, {\em Handbook of
  Set-Theoretic Topology}, pages 111--167. North-Holland, Amsterdam, 1984.

\bibitem{Zapletal2008}
Jindrich Zapletal.
\newblock {\em Forcing Idealized}.
\newblock Cambridge Tracts in Mathematics. Cambridge University Press, 2008.

\end{thebibliography}

\end{document}